\documentclass[12pt,reqno]{amsart}
\usepackage{times}

\newtheorem{thm}{Theorem}

\newcommand{\C}{\mathbb C}
\newcommand{\inc}{\int_{\C}}
\newcommand{\fpa}{F^p_\alpha}

\newcommand{\CP}{{\mathcal S}}
\newcommand{\CC}{{\mathcal C}}

\begin{document}

\title[Circle Packing and Fock Spaces]
{Circle packing and interpolation in Fock spaces}

\author{Daniel Stevenson and Kehe Zhu}
\address{Department of Mathematics and Statistics, State University of New York, 
Albany, NY 12222, USA.}
\email{dstevenson@albany.edu, kzhu@math.albany.edu}

\date\today
\keywords{Circle packing, circle covering, Fock spaces, interpolating sequence, 
sampling sequence.}
\subjclass[2010]{30H20 and 52C26.}

\begin{abstract}
It was shown by James Tung in 2005 that if a sequence $Z=\{z_n\}$ of points in the
complex plane satisfies 
$$\inf_{n\not=m}|z_n-z_m|>2/\sqrt\alpha,$$
then $Z$ is a sequence of interpolation for the Fock space $F^p_\alpha$. Using 
results from circle packing, we show that the constant above can be improved to 
$$\sqrt{2\pi/(\sqrt3\,\alpha)},$$
which is strictly smaller than $2/\sqrt\alpha$. A similar result will also be obtained 
for sampling sequences.
\end{abstract}

\maketitle

\section{Introduction}

Let $\C$ denote the complex plane and $dA$ denote area measure on $\C$. For $0<p\le\infty$
and $\alpha>0$ let $\fpa$ denote the space of all entire functions $f$ such that
$$f(z)e^{-\frac\alpha2|z|^2}\in L^p(\C,dA).$$
The spaces $\fpa$ are called Fock spaces.

When $0<p<\infty$, we write
$$\|f\|_{p,\alpha}=\left[\frac{p\alpha}{2\pi}\inc\left|f(z)e^{-\frac\alpha2|z|^2}
\right|^p\,dA(z)\right]^{\frac1p}$$
for $f\in\fpa$. When $p=\infty$, we write
$$\|f\|_{\infty,\alpha}=\sup\{|f(z)|e^{-\frac\alpha2|z|^2}:z\in\C\}$$
for $f\in F^\infty_\alpha$. See \cite{Z} for more information about Fock spaces.

An important concept in the theory of Fock spaces is the notion of interpolating and 
sampling sequences. More specifically, a sequence $Z=\{z_n\}$ of distinct points in the
complex plane is called an interpolating sequence for $\fpa$ if for every sequence
$\{v_n\}$ of complex values satisfying
$$\left\{v_ne^{-\frac\alpha2|z_n|^2}\right\}\in l^p$$
there exists a function $f\in\fpa$ such that $f(z_n)=v_n$ for all $n$. Similarly,
$Z=\{z_n\}$ is called a sampling sequence for $\fpa$ if there exists a positive 
constant $C$ such that
$$C^{-1}\|f\|_{p,\alpha}\le\left\|\left\{f(z_n)e^{-\frac\alpha2|z_n|^2}\right\}
\right\|_{l^p}\le C\|f\|_{p,\alpha}$$
for all $f\in\fpa$.

Interpolating and sampling sequences for Fock spaces are characterized by Seip and Wallsten
in \cite{S,SW}. Their characterizations are based on a special notion of density for
sequences in the complex plane. More specifically, if
$$B(z,r)=\{w\in\C:|z-w|<r\}$$
is the Euclidean disk centered at $z$ with radius $r$, and if $n(Z,B(z,r))$ denotes the
number of points in $Z\cap B(z,r)$, then we define
$$D^-(Z)=\liminf_{r\to\infty}\inf_{\zeta\in\C}\frac{n(Z,B(\zeta,r))}{\pi r^2},$$
and
$$D^+(Z)=\limsup_{r\to\infty}\sup_{\zeta\in\C}\frac{n(Z,B(\zeta,r))}{\pi r^2},$$
and call them the lower and upper densities of $Z$, respectively. 

Note that the sequence $Z=\{z_n\}$ is said to be separated if there exists a positive 
constant $\delta$ such that $|z_n-z_m|\ge\delta$ for all $n\not=m$.

\begin{thm}[Seip-Wallsten]
Suppose $Z$ is a separated sequence and $0<p\le\infty$. Then $Z$ is interpolating for
$\fpa$ if and only if $D^+(Z)<\alpha/\pi$; and $Z$ is sampling for $\fpa$ if and only if
$D^-(Z)>\alpha/\pi$.
\label{1}
\end{thm}

Roughly speaking, a sequence of points in the complex plane is interpolating for $\fpa$ 
if it is sparse enough. Similarly, a sequence is sampling for $\fpa$ if it is 
sufficiently dense. Based on Seip and Wallsten's theorem above, James Tung obtained 
the following result in \cite{T,T1}.

\begin{thm}[Tung]
If $Z=\{z_n\}$ is a sequence of points in the complex plane satisfying
$$\inf_{n\not=m}|z_n-z_m|>\frac2{\sqrt\alpha},$$
then $Z$ is an interpolating sequence for $\fpa$.
\label{2}
\end{thm}

Tung's result gives us an easily verifiable sufficient condition for a sequence to be
interpolating for $\fpa$. The purpose of this note is to show that Tung's result can be
improved, namely, we will prove that the constant in Tung's theorem above can be improved to
$$\sqrt{\frac{2\pi}{\sqrt3\,\alpha}},$$
which is strictly smaller than $2/\sqrt\alpha$.

We will also obtain a similar result about sampling sequences for $\fpa$. Our approach is
based on some classical results for circle packing in the complex plane.

\section{Circle packing}

A circle packing of the plane is a countable collection of non-overlapping circles in $\C$.
For many years it had remained a curiosity to mathematicians as to what arrangement of
circles of a fixed radius would cover the largest proportion of the plane. 

It was already known to Joseph Louis Lagrange in 1773 that, among lattice arrangements 
of circles, the highest density is achieved by the hexagonal lattice of the bee's honeycomb, 
in which the centers of the circles form a hexagonal lattice, with each circle surrounded 
by 6 others. The density of such a packing is given by
$$\frac\pi{\sqrt{12}}=0.9069\cdots.$$
In 1890, Axel Thue showed that this density was actually maximal among all possible circle 
packings (not necessarily lattice packings). But his proof was considered to be incomplete 
by some mathematicians, and a more rigorous proof was finally found by 
L\'aszl\'o Fejes T\'oth in 1940.

Let $S(z,r)=\partial B(z,r)$ denote the circle centered at $z$ with radius $r$. If
$\CP=\{S(z_n,r_0)\}$ is a circle packing in the plane, its packing density is defined as
$$\Delta(\CP)=\limsup_{r\to\infty}\sup_{\zeta\in\C}\frac1{\pi r^2}
\sum\left\{\pi r_0^2: B(z_n,r_0)\cap B(\zeta,r)\not=\emptyset\right\}.$$
See page 22 of \cite{R}. Therefore, the historical result about circle packing in the plane 
can be stated as follows. See page 1 of \cite{R} for example.

\begin{thm}
For any circle packing $\CP$ we always have 
$$\Delta(\CP)\le\frac\pi{\sqrt{12}}=\frac\pi{2\sqrt3}<1.$$ 
Furthermore, equality is achieved by the hexagonal packing.
\label{3}
\end{thm}

There is also a corresponding notion of circle covering. More specifically, we say
that a countable collection of circles $\CC=\{S(z_n,r_0)\}$ is a circle covering of the
plane if the union of $\{B(z_n,r_0)\}$ covers the whole plane $\C$. The number
$$\delta(\CC)=\liminf_{r\to\infty}\inf_{\zeta\in\C}\frac1{\pi r^2}
\sum\left\{\pi r_0^2:B(z_n,r_0)\subset B(\zeta,r)\right\}$$
will be called the covering density of $\CC$. See page 22 of \cite{R}. The following 
theorem is a classical result from circle covering in the plane. See page 16 of \cite{R}
for example.

\begin{thm}
For any circle covering $\CC$ we always have 
$$\delta(\CC)\ge\frac{2\pi}{3\sqrt3}>1.$$
Furthermore, equality is achieved by circles centered at any hexagonal lattice with 
the same radius chosen to be the minimum so that these circles cover the plane.
\label{4}
\end{thm}

There is a large body of research work concerning circle packing and circle covering.
See \cite{R} for an elementary introduction to these topics and \cite{St} for a modern
treatment of the subject based on analytic functions.

\section{Interpolation and sampling in Fock spaces}

We now apply the classical results about circle packing and circle covering to obtain
sufficient conditions for interpolating and sampling sequences for Fock spaces.

\begin{thm}
If $Z=\{z_n\}$ is a sequence of points in the complex plane and
$$\inf_{n\not=m}|z_n-z_m|>\sqrt{\frac{2\pi}{\sqrt3\,\alpha}},$$
then $Z$ is an interpolating sequence for $\fpa$.
\label{5}
\end{thm}

\begin{proof}
Suppose $\sigma>0$ and $|z_n-z_m|\ge\sigma$ for all $n\not=m$. It is then clear that
$\CP=\{S(z_n,\sigma/2)\}$ is a circle packing in the complex plane. By Theorem~\ref{3},
$\Delta(\CP)\le\pi/\sqrt{12}$. It follows from the definition of packing density
that for any $\varepsilon>0$ there exists some positive number $R$ such that for all 
$r>R$ and all $\zeta\in\C$ we have
$$\frac1{\pi r^2}\sum\left\{\frac{\pi\sigma^2}4:B(z_n,\sigma/2)\cap B(\zeta,r)
\not=\emptyset\right\}<\frac\pi{\sqrt{12}}+\varepsilon.$$
Since $z_n\in B(\zeta,r)$ implies
$$B(z_n,\sigma/2)\cap B(\zeta,r)\not=\emptyset,$$
we must also have
$$\frac1{\pi r^2}\sum\left\{\frac{\pi\sigma^2}4:z_n\in B(\zeta,r)\right\}<
\frac\pi{\sqrt{12}}+\varepsilon.$$
Rewrite this as
$$\frac{\sigma^2}{4r^2}\,n(Z\cap B(\zeta,r))<\frac\pi{\sqrt{12}}+\varepsilon,$$
or equivalently,
$$\frac{n(Z\cap B(\zeta,r))}{\pi r^2}<\frac4{\sqrt{12}\,\sigma^2}+
\frac{4\varepsilon}{\sigma^2\pi}.$$
Take the supremum over $\zeta\in\C$ and let $r\to\infty$. We obtain
$$D^+(Z)\le\frac2{\sqrt3\,\sigma^2}+\frac{4\varepsilon}{\sigma^2\pi}.$$
Since $\varepsilon$ is arbitrary, we must have $D^+(Z)\le2/(\sqrt3\,\sigma^2)$.

Now if
$$\sigma>\sqrt{\frac{2\pi}{\sqrt3\,\alpha}},$$
then
$$\frac2{\sqrt3\,\sigma^2}<\frac\alpha\pi,$$
so that $D^+(Z)<\alpha/\pi$. Combining this with Theorem~\ref{1}, we conclude that the 
condition
$$\inf_{n\not=m}|z_n-z_m|>\sqrt{\frac{2\pi}{\sqrt3\,\alpha}}$$
implies that $Z=\{z_n\}$ is a sequence of interpolation for $\fpa$.
\end{proof}

Suppose $Z=\{z_n\}$ is a hexagonal lattice and $\sigma$ is the distance from any point in
$Z$ to its nearest neighbor. When $r$ is very large, the difference between the number of 
points $z_n$ satisfying $B(z_n,\sigma/2)\cap B(\zeta,r)\not=\emptyset$ and the number of
points $z_n$ satisfying $z_n\in B(\zeta,r)$ is insignificant. Since the hexagonal circle
packing has the largest packing density, a careful examination of the proof above shows 
that the constant $\sqrt{(2\pi)/(\sqrt3\,\alpha)}$ in Theorem~\ref{5} is best possible.

A companion result for sampling sequences is the following.

\begin{thm}
Let $Z=\{z_n\}$ be a sequence of distinct points in the complex plane. If there exists
a positive number 
$$\sigma<\sqrt{\frac{2\pi}{3\sqrt3\,\alpha}}$$ 
such that $\CC=\{S(z_n,\sigma)\}$ is a circle covering for $\C$, then $Z$ is a 
sampling sequence for $\fpa$.
\label{6}
\end{thm}

\begin{proof}
Suppose that $\CC=\{S(z_n,\sigma)\}$ is a circle covering in the complex plane. By
Theorem~\ref{4}, we have $\delta(\CC)\ge(2\pi)/(3\sqrt3)$. It follows from the definition
of covering density that for any $\varepsilon>0$ there exists a positive number $R$ 
such that for all $r>R$ and all $\zeta\in\C$ we have
$$\frac1{\pi r^2}\sum\left\{\pi\sigma^2:B(z_n,\sigma)\subset B(\zeta,r)\right\}
\ge\frac{2\pi}{3\sqrt3}-\varepsilon.$$
Since $B(z_n,\sigma)\subset B(\zeta,r)$ implies that $z_n\in B(\zeta,r)$, we must also have
$$\frac1{\pi r^2}\sum\left\{\pi\sigma^2:z_n\in B(\zeta,r)\right\}\ge
\frac{2\pi}{3\sqrt3}-\varepsilon.$$
Rewrite this as
$$\frac{\pi\sigma^2}{\pi r^2}\,n(Z\cap B(\zeta,r))\ge\frac{2\pi}{3\sqrt3}-\varepsilon,$$
or equivalently,
$$\frac{n(Z\cap B(\zeta,r))}{\pi r^2}\ge\frac2{3\sqrt3\,\sigma^2}
-\frac\varepsilon{\pi\sigma^2}.$$
Take the infimum over $\zeta$ and let $r\to\infty$. We obtain
$$D^-(Z)\ge\frac2{3\sqrt3\,\sigma^2}-\frac\varepsilon{\pi\sigma^2}.$$
Since $\varepsilon$ is arbitrary, we must have
$$D^-(Z)\ge\frac2{3\sqrt3\,\sigma^2}.$$
It is then easy to see that the condition
$$\sigma<\sqrt{\frac{2\pi}{3\sqrt3\,\alpha}}$$
implies $D^-(Z)>\alpha/\pi$. This along with Theorem~\ref{1} shows that $Z=\{z_n\}$ is
a sampling sequence for $\fpa$.
\end{proof}

Again, if $\CC=\{S(z_n,\sigma)\}$ is an optimal hexagonal circle covering of the complex
plane, then for very large $r$, the difference between the number of points $z_n$ 
satisfying $B(z_n,\sigma)\subset B(\zeta,r)$ and the number of points $z_n$ satisfying
$z_n\in B(\zeta,r)$ is negligible. Therefore, the constant $\sqrt{(2\pi)/(3\sqrt3\,\alpha)}$
in Theorem~\ref{6} is best possible.

\end{document}